\documentclass [12pt]{amsart}

\usepackage{amsmath}
\usepackage{amsthm}
\usepackage{amssymb}
\usepackage{amscd}
\usepackage{xspace}
\usepackage{verbatim}
\newtheorem{theorem}{Theorem}[section]
\newtheorem{corollary}{Corollary}[section]
\newtheorem{lemma}{Lemma}[section]

\theoremstyle{definition}
\newtheorem{definition}{Definition}[section]

\theoremstyle{remark}
\newtheorem{remark}{Remark}[section]
\numberwithin{equation}{section}

\newcommand{\e}{\varepsilon}

\renewcommand{\O}{\Omega}

\renewcommand{\liminf}{\varliminf}
\renewcommand{\limsup}{\varlimsup}

\newcommand{\field}[1]{\mathbb{#1}}

\newcommand{\R}{\field{R}}

\newcommand{\er}{\eqref}

\DeclareMathOperator{\supp}{supp}

\renewcommand{\O}{\Omega}

\newcommand{\f}{\varphi}

\begin{document}
\title{On a variational approach to the Navier-Stokes Equations\footnote{My old unpublished paper
(2007)}}
\maketitle
\begin{center}
\textsc{Arkady Poliakovsky \footnote{E-mail:
poliakov@math.bgu.ac.il}
}\\[3mm]
Department of Mathematics, Ben Gurion University of the Negev,\\
P.O.B. 653, Be'er Sheva 84105, Israel
\\[2mm]
\end{center}
\section{Introduction}
Let $\O\subset\R^N$ be a domain. The initial-boundary value problem
for the incompressible Navier-Stokes Equations is the following one,
\begin{equation}\label{IBNS}
\begin{cases}\;(i)\;\;\,\frac{\partial v}{\partial t}+div_x\,(v\otimes
v)+\nabla_x p=\nu\Delta_x v+f\quad\quad
\forall(x,t)\in\O\times(0,T)\,,\\\,(ii)\;\; div_x\,
v=0\quad\quad\forall(x,t)\in\O\times(0,T)\,,
\\(iii)\;\;v=0\quad\quad\forall(x,t)\in\partial\O\times(0,T)\,,
\\(iv)\,\;\;v(x,0)=v_0(x)\quad\quad\forall x\in\O\,.
\end{cases}
\end{equation}
Here $v=v(x,t):\O\times(0,T)\to\R^N$ is an unknown velocity,
$p=p(x,t):\O\times(0,T)\to\R$ is an unknown pressure, associated
with $v$, $\nu>0$ is a given constant viscosity,
$f:\O\times(0,T)\to\R^N$ is a given force field and $v_0:\O\to\R^N$
is a given initial velocity. The existence of weak solution to
\er{IBNS} satisfying the Energy inequality was first proved in the
celebrating works of Leray (1934). There are many different
procedures for constructing weak solutions (see Leray
\cite{Ler1},\cite{Ler2} (1934); Kiselev and Ladyzhenskaya \cite{KaL}
(1957); Shinbrot \cite{Sh} (1973)). The most common methods are
based on the so called Faedo-Galerkin approximation process.
Application of Faedo-Galerkin method for \er{IBNS} was first
considered by Hopf in \cite{Hopf}. We also refer to Masuda~\cite{Masuda} for the problem in higher dimension. In this paper we present a
variational method to investigate the Navier-Stokes equations that we
thought to be completely new, see however the remarks below.
As an application of this method we give a relatively simple
proof of the existence of weak solutions to the problem \er{IBNS}.

 Let us briefly describe our method. Consider for simplicity $f=0$ in
\er{IBNS}. For every smooth $u:\bar\O\times[0,T]\to\R^N$ satisfying
conditions $(ii)-(iv)$ of \er{IBNS} define the energy functional
\begin{equation}\label{energ}
E(u):=\frac{1}{2}\int_0^T\int_\O\Big(\nu|\nabla_x
u|^2+\frac{1}{\nu}|\nabla_x \bar
H_u|^2\Big)\,dxdt+\frac{1}{2}\int_\O|u(x,T)|^2\,dx\,,
\end{equation}
where $\bar H_u(x,t)$ solves the following Stokes system for every
$t\in(0,T)$,
\begin{equation}\label{stokes}
\begin{cases}\Delta_x \bar H_u=\Big(\frac{\partial u}{\partial
t}+div_x\,(u\otimes u)\Big)+\nabla_x p\quad\quad x\in\O\,,\\
div_x\,\bar H_u=0\quad\quad x\in\O\,,\\ \bar H_u=0\quad\quad\forall
x\in\partial\O\,.\end{cases}
\end{equation}
A simple integration by parts gives
\begin{equation}\label{energ1}
E(u)=\frac{1}{2\nu}\int_0^T\int_\O\Big(|\nu\nabla_x u-\nabla_x \bar
H_u|^2\Big)\,dxdt+\frac{1}{2}\int_\O|v_0(x)|^2\,dx\,.
\end{equation}
Therefore, if there exists at least a smooth solution to \er{IBNS}
(with $f=0$) then a smooth function $u:\O\times(0,T)\to\R^N$ will be
a solution to \er{IBNS} (with $f=0$) if and only if it is a
minimizer of the functional in \er{energ} among all smooth
divergence free vector fields satisfying the boundary and the
initial value conditions of \er{IBNS}. For the rigorous formulations
and statements, see Section \ref{odin}. This remark relates the
problem of existence of solutions of the Navier-Stokes equations to
that  of minimizing the energy $E(u)$.

 Unfortunately, when applying this method to the Navier-Stokes
Equation one
meets certain difficulties, for example in proving the existence of
minimizers to $E$. But we can apply this method to a suitable
approximation of problem \er{IBNS}. We approximate \er{IBNS} by
replacing the nonlinear term $div_x(v\otimes v)$ with the terms
$div_x\big\{f_n(|v|^2)(v\otimes v)\big\}$, where $f_n:\R^+\to\R^+$
are regular cutoff functions satisfying $f_n(s)=1$ for $s\leq n$ and
$f_n(s)=0$ for $s>2n$. The approximating problems are simpler than
\er{IBNS}, since the nonlinear term has higher integrability. Next
we consider the energies $E_n$ corresponding to the approximating
problems and investigate the Euler-Lagrange equations of $E_n$ and
the existence of minimizers. In this way we get solutions to the
approximating problems which satisfy the energy equality (in fact
these solutions will be regular if the initial data and the domain
are). Next we pass to the limit for $n\to\infty$ and obtain a weak
solution to \er{IBNS}. For the details see Section \ref{dva}.

 After completing the first version of this paper I learned that an
energy-functional, very similar to \er{energ}, was used by Ghoussoub
and his coauthors, see \cite{Gho,GosM,GosM1,GosTz},
 to prove existence of weak
solutions for \er{IBNS} and to study many other evolution equations.
The basic variational principle behind this approach was first
introduced by Brezis and Ekeland, see \cite{Brez} (I wasn't aware of
this article as well).
 The main new feature of our method is that, unlike the previous works
 mentioned above,
 we menage to deduce directly from the Euler-Lagrange equation associated with
 \eqref{energ} that the minimizer is a solution of the original
 problem \eqref{IBNS}.

We shall now demonstrate our method in  the simple
example of the heat equation. In this case, the energy-functional
takes
the form
\begin{equation}\label{energht}
\bar E(u):=\frac{1}{2}\int_0^T\int_\O\Big(|\nabla_x
u|^2+\big|\nabla_x\{\Delta_x^{-1}(\partial_t
u)\}\big|^2\Big)\,dxdt+\frac{1}{2}\int_\O|u(x,T)|^2\,dx\,,
\end{equation}
where $\Delta^{-1}f$ is the solution of
\begin{equation*}
\begin{cases}\Delta y=f\quad\quad x\in\O\,,\\ y=0\quad\quad\forall
x\in\partial\O\,.\end{cases}
\end{equation*}
The variational functional of type \er{energht} was first considered
by Brezis and Ekeland~\cite{Brez} in the more general case of
gradient flows. Let us investigate the Euler-Lagrange equation for
\er{energht}. If $u$ satisfies $u(x,t)=0$ for every
$(x,t)\in\partial\O\times(0,T)$ and $u(x,0)=v_0(x)$, then, as before,
\begin{equation*}
\bar E(u):=\frac{1}{2}\int_0^T\int_\O\Big(\big|\nabla_x\{
u-\Delta_x^{-1}(\partial_t u)\}\big|^2\Big)\,dxdt+\frac{1}{2}\int_\O
v_0^2(x)\,dx\,,
\end{equation*}
Set $W_u:=u-\Delta_x^{-1}(\partial_t u)$. Then, for every minimizer
$u$ and for every smooth test function $\delta(x,t)$ satisfying
$\delta(x,t)=0$ for every $(x,t)\in\partial\O\times(0,T)$ and
$\delta(x,0)=0$, we obtain
\begin{multline*}
0=\frac{d \bar E(u+s\delta)}{ds}\Big|_{(s=0)}=\lim\limits_{s\to
0}\frac{1}{2s}\int_0^T\int_\O\Big(|\nabla_x
W_{(u+s\delta)}|^2-|\nabla_x W_{u}|^2\Big)=\\
-\lim\limits_{s\to 0}\frac{1}{2s}\int_0^T\int_\O\big(\Delta_x
W_{(u+s\delta)}-\Delta_x W_{u}\big)\cdot\big(
W_{(u+s\delta)}+W_{u}\big)=\\
\lim\limits_{s\to 0}\frac{1}{2}\int_0^T\int_\O\big(-\Delta_x
\delta+\partial_t\delta\big)\cdot\big( W_{(u+s\delta)}+W_{u}\big)=
\int_0^T\int_\O\big(\nabla W_{u}\cdot\nabla_x
\delta+W_{u}\cdot\partial_t\delta\big)\,.
\end{multline*}
Since $\delta$ was arbitrary (in particular $\delta(x,T)$ is free)
we deduce that $\Delta_x W_u+\partial_t W_u=0$, $W_u(x,T)=0$ and
$W_u=0$ if $x\in\partial\O$. Changing variables $\tau:=T-t$ gives
\begin{equation*}
\begin{cases}
\partial_\tau W_u=\Delta_x W_u\quad\quad\forall (x,\tau)\in\O\times(0,T)\,,\\
W_u(x,0)=0\,,\\
W_u(x,\tau)=0\quad\quad\quad\forall
(x,\tau)\in\partial\O\times(0,T)\,.
\end{cases}
\end{equation*}
Therefore $W_u=0$ and then $\Delta_x u=\partial_t u$, i.e., $u$ is the solution of
the heat equation.

\section{Preliminaries}
For two matrices $A,B\in\R^{p\times q}$  with $ij$-th entries
$a_{ij}$ and $b_{ij}$ respectively, we write
$A:B\,:=\,\sum\limits_{i=1}^{p}\sum\limits_{j=1}^{q}a_{ij}b_{ij}$.\\
Given a vector valued function
$f(x)=\big(f_1(x),\ldots,f_k(x)\big):\O\to\R^k$ ($\O\subset\R^N$) we
denote by $\nabla_x f$ the $k\times N$ matrix with
$ij$-th entry $\frac{\partial f_i}{\partial x_j}$.\\
For a matrix valued function $F(x):=\{F_{ij}(x)\}:\R^N\to\R^{k\times
N}$ we denote by $div\,F$ the $\R^k$-valued vector field defined by
$div\,F:=(l_1,\ldots,l_k)$ where
$l_i=\sum\limits_{j=1}^{N}\frac{\partial F_{ij}}{\partial x_j}$.
%
%
%
%
%
%
%
%
Throughout the rest of the paper we assume that $\O$ is domain in
$\R^N$.
\begin{definition}\label{dXY}
We denote:
\begin{itemize}
\item
By $\mathcal{V}_N$ the space $\{\f\in
C_c^\infty(\O,\R^N):\,div\,\f=0\}$ and by $L_N$ the space, which is
the closure of $\mathcal{V}_N$ in the space $L^2(\O,\R^N)$, endowed
with
the norm $\|\f\|:=\big(\int_\O|\f|^2dx\big)^{1/2}$.

\item
By $\bar H^1_0(\O,\R^N)$ the closure of $C_c^\infty(\O,\R^N)$ with
respect to the norm
$|||\f|||:=\big(\int_\O|\nabla\f|^2dx\big)^{1/2}$.
This space differ from $H^1_0(\O,\R^N)$ only in the case of
unbounded domain.

\item
By $V_N$ the closure of $\mathcal{V}_N$ in $\bar H^1_0(\O,\R^N)$.

\item
By $V_N^{-1}$ the space dual to $V_N$.
%
%
%
%
%
%

\item
By $\mathcal{Y}$ the space
$$\mathcal{Y}:=\{\f(x,t)\in
C^\infty_c(\O\times[0,T],\R^N):\,div_x\,\f=0\}\,.$$
\end{itemize}
\end{definition}

\begin{remark}\label{rem}
It is obvious that $u\in\mathcal{D}'(\O,\R^N)$ (rigorously the
equivalence class of $u$, up to gradients) belongs to $V_N^{-1}$ if
and only if there exists $w\in V_N$ such that
$$\int_\O\nabla w:\nabla \delta\, dx=
-<u,\delta>\quad\quad\forall\delta\in V_N\,.$$ In particular $\Delta
w=u+\nabla p$ as a distribution and
$$|||w|||=\sup\limits_{\delta\in V_N,\;|||\delta|||\leq
1}<u,\delta>=|||u|||_{-1}\,.$$
\end{remark}
\begin{definition}\label{dffgghhjjkt}
We will say that the distribution
$l\in\mathcal{D}'(\O\times(0,T),\R^N)$ belongs to
$L^2(0,T;V_N^{-1})$, if there exists $v(\cdot,t)\in
L^2(0,T;V_N^{-1})$, such that for every $\psi(x,t)\in
C^\infty_c(\O\times(0,T),\R^N)$, satisfying $div_x\,\psi=0$, we have
$$<l(\cdot,\cdot),\psi(\cdot,\cdot)>=\int_0^T<v(\cdot,t),\psi(\cdot,t)>\,dt\,.$$
\end{definition}
\begin{remark}\label{rem1}
Let $v(\cdot,t)\in L^2(0,T;V_N^{-1})$. For a.e. $t\in[0,T]$ consider
$V_v(\cdot,t)$ as in Remark \ref{rem}, corresponding to
$v(\cdot,t)$, i.e.
$$\int_\O\nabla_x V_v(x,t):\nabla_x \delta(x)\, dx=
-<v(\cdot,t),\delta(\cdot)>\quad\quad\forall\delta\in V_N\,.$$ Then
it is clear that $V_v(\cdot,t)\in L^2(0,T;V_N)$ and
$$\|V_v\|_{L^2(0,T;V_N)}=\|v\|_{L^2(0,T;V_N^{-1})}\,.$$
\end{remark}
In the sequel we will need several lemmas. In all of them
$\O\subset\R^N$ is a bounded domain. The following Lemma can be
proved in the same way as Lemmas 2.1 and 2.2 in \cite{Gal}.
\begin{lemma}\label{vbnhjjm}
Let $u\in L^2(0,T;V_N)\cap L^\infty(0,T;L_N)$ be such that
$\partial_t u\in L^2(0,T;V_N^{-1})$. Consider $V_0(\cdot,t)\in
L^2(0,T;V_N)$ as in Remark \ref{rem1}, corresponding to $\partial_t
u$. Then we can redefine $u$ on a subset of $[0,T]$ of Lebesgue
measure zero, so that
%
%
$u(\cdot,t)$ will be $L_N$-weakly continuous in $t$ on $[0,T]$.
Moreover, for every $0\leq a<b\leq T$ and for every
$\psi(x,t)\in\mathcal{Y}$ (see Definition \ref{dXY})
we will have
\begin{multline}\label{eqmult}
\int_a^b\int_\O\nabla_x V_0:\nabla_x\psi\,
dxdt-\int_a^b\int_\O u\cdot\partial_t\psi\,dxdt\\
=\int_\O u(x,a)\cdot\psi(x,a)dx-\int_\O u(x,b)\cdot\psi(x,b)dx\,.
\end{multline}
\end{lemma}
\begin{remark}\label{rem3}
Let $F\in Lip\,(\R^N,\R^{N\times N})$ satisfying $F(0)=0$. Then for
every $u\in L^\infty(0,T;L_N)$ we have $F(u)\in
L^\infty\big(0,T;L^2(\O,\R^{N\times N})\big)$ and therefore $div_x\,
F(u)\in L^2(0,T;V_N^{-1})$. If in addition $\partial_t u\in
L^2(0,T;V_N^{-1})$ then we obtain $\partial_t u+div_x\, F(u)\in
L^2(0,T;V_N^{-1})$.
\end{remark}
We have then the following Corollary to Lemma \ref{vbnhjjm}.
\begin{corollary}\label{colkk}
Let $u$ be as in Lemma \ref{vbnhjjm} and let $F\in
Lip\,(\R^N,\R^{N\times N})$ satisfying $F(0)=0$. Assume, in
addition, that $u(\cdot,t)$ is $L_N$-weakly continuous in $t$ on
$[0,T]$ (see Lemma \ref{vbnhjjm}). Consider $V(\cdot,t)\in
L^2(0,T;V_N)$ as in Remark \ref{rem1}, corresponding to $\partial_t
u+div_x\, F(u)$. Then for every $0\leq a<b\leq T$ and for every
$\psi(x,t)\in\mathcal{Y}$ we have
\begin{multline}\label{eqmultnep}
\int_a^b\int_\O\nabla_x V:\nabla_x\psi\,
dxdt-\int_a^b\int_\O\big(u\cdot\partial_t\psi+F(u):\nabla_x\psi\big)\,dxdt\\
=\int_\O u(x,a)\cdot\psi(x,a)dx-\int_\O u(x,b)\cdot\psi(x,b)dx\,.
\end{multline}
\end{corollary}
We will need in the sequel the following compactness result.
\begin{lemma}\label{ComTem1}
Let $\{u_n\}\subset L^2(0,T;V_N)\cap L^\infty(0,T;L_N)$ be a
subsequence, bounded in $L^\infty(0,T;L_N)$ and such that
\begin{equation}\label{waeae1stggg}
u_n\rightharpoonup u_0 \quad\text{weakly in } L^2(0,T;V_N)\,,
\end{equation}
and
\begin{equation}\label{waeae12}
u_n(\cdot,t)\rightharpoonup u_0(\cdot,t) \quad\text{weakly in }
L_N\quad\forall t\in(0,T)\,.
\end{equation}
Then
\begin{equation}\label{staeae1gllluk}
u_n\to u_0 \quad\text{strongly in } L^2(0,T;L_N)\,.
\end{equation}
\end{lemma}
We will give the proof of this Lemma in the Appendix.

\section{Existence of the weak solution to the Navier-Stokes
Equations}\label{dva}
Throughout this section we assume that
$\O\subset\R^N$ is a bounded domain.
\begin{definition}\label{deffffggh}
Let $F(v)=\{F_{ij}(v)\}\in C^1(\R^N,\R^{N\times N})\cap Lip$ satisfy
$F(0)=0$ and $\frac{\partial F_{ij}}{\partial v_m}(v)=\frac{\partial
F_{mj}}{\partial v_i}(v)$ for all $v\in\R^N$ and
$m,i,j\in\{1,\ldots,N\}$. Denote the class of all such $F$ by
$\mathfrak{F}$.
\end{definition}
\begin{remark}\label{rem9888}
Let $F\in\mathfrak{F}$. Then it is clear that there exists
$G(v)=(G_1(v),\ldots, G_N(v))\in C^2(\R^N,\R^N)$, such that
$\frac{\partial G_{j}}{\partial v_i}(v)=F_{ij}(v)$ i.e. $\nabla_v
G(v)=(F(v))^T$.
\end{remark}
Using our variational approach, we will prove in the sequel the
existence of a solution of the following problem
\begin{equation}\label{apprxnn}
\begin{cases}\frac{\partial v}{\partial t}+div_x\,F(v)+\nabla_x p=\Delta_x v\quad\quad
\forall(x,t)\in\O\times(0,T)\,,\\ div_x\,
v=0\quad\quad\forall(x,t)\in\O\times(0,T)\,,
\\v=0\quad\quad\forall(x,t)\in\partial\O\times(0,T)\,,
\\ v(x,0)=v_0(x)\quad\quad\forall x\in\O\,,
\end{cases}
\end{equation}
for every $F\in\mathfrak{F}$, which in addition satisfies the Energy
Equality (see Theorem \ref{premain}). But first of all, in the proof
of the following theorem we would like to explain how this fact
implies the existence of weak solution to the Navier-Stokes
Equation.
\begin{theorem}\label{WeakNavierStokes}
Let $v_0(x)\in L_N$.
Then there exists $u\in L^2(0,T;V_N)\cap L^\infty(0,T;L_N)$
satisfying
\begin{multline}\label{nnhogfc48}
\int_\O v_0(x)\cdot\psi(x,0)\,dx+\int_0^T\int_\O
\big(u\cdot\partial_t\psi+(u\otimes u):\nabla_x
\psi\big)=\int_0^T\int_\O\nabla_x u:\nabla_x\psi\,,
\end{multline}
for every $\psi(x,t)\in C^\infty_c(\O\times[0,T),\R^N)$ such that
$div_x\,\psi=0$, i.e.
$$\Delta_x u=\partial_t u+div_x\,(u\otimes u)+\nabla_x p
\,, \quad\text{and } u(x,0)=v_0(x)\,.$$ Moreover, for a.e.
$\tau\in[0,T]$ we have
\begin{equation}\label{yteqin}
\int_0^\tau\int_\O|\nabla_x u|^2\,dxdt
\leq\frac{1}{2}\bigg(\int_\O v_0^2(x)dx -\int_\O
u^2(x,\tau)dx\bigg)\,.
\end{equation}
%
%
%
%
%
%
%
%
\end{theorem}
\begin{proof}
Fix some $h(s)\in C^\infty(\R,[0,1])$, satisfying $h(s)=1$ $\forall
s\leq 1$ and $h(s)=0$ $\forall s\geq 2$. For every $n\in\mathbb{N}$
define $f_n(s):=h(s/n)$. Consider
\begin{equation}\label{klnlnkmlm}
F_n(v):=f_n(|v|^2)(v\otimes v)+g_n(|v|^2)I_N\,,
\end{equation}
where $I_N$ is a $N\times N$-unit matrix and
$g_n(r):=\frac{1}{2}\int_0^rf_n(s)ds$. Then for every $n$ we have
$F_n\in\mathfrak{F}$ and there exists $A>0$ such that $|F_n(v)|\leq
A|v|^2$ for every $v$ and $n$. Fix also some sequence
$\{v_{0}^{(n)}\}_{n=1}^{\infty}\subset\mathcal{V}_N$ such that
$v_{0}^{(n)}\to v_0$ strongly in $L_N$ as $n\to\infty$. By Theorem
\ref{premain}, bellow, for every $n$ there exist a function $u_n\in
L^2(0,T;V_N)\cap L^\infty(0,T;L_N)$, such that $\partial_t u_n\in
L^2(0,T;V_N^{-1})$ and $u_n(\cdot,t)$ is $L_N$-weakly continuous in
$t$ on $[0,T]$, which satisfy
\begin{multline}\label{nnhogfc48nnj}
\int_\O v_{0}^{(n)}(x)\cdot\psi(x,0)+\int_0^T\int_\O
\big(u_n\cdot\partial_t\psi+F_n(u_n):\nabla_x
\psi\big)=\int_0^T\int_\O\nabla_x u_n
:\nabla_x\psi\,,
\end{multline}
for every $\psi(x,t)\in C^\infty_c(\O\times[0,T),\R^N)$, such that
$div_x\,\psi=0$. Moreover, by the same Theorem, for every
$\tau\in[0,T]$ we obtain
\begin{equation}\label{yteqinnnj1}
\frac{1}{2}\int_\O u_n^2(x,\tau)dx+\int_0^\tau\int_\O|\nabla_x
u_n|^2\,dxdt
=\frac{1}{2}
\int_\O(v_{0}^{(n)})^2(x)dx\,.
\end{equation}
Therefore, since $v_{0}^{(n)}$ is bounded in $L_N$ we obtain that
there exists $C>0$ independent of $n$ and $t$ such that
\begin{equation}\label{bonmdf}
\|u_n(\cdot,t)\|_{L_N}\leq C\quad\forall n\in\mathbb{N}, t\in[0,T].
\end{equation}
Moreover, $\{u_n\}$ is bounded in $L^2(0,T;V_N)$. By
\er{nnhogfc48nnj} and \er{eqmultnep}, for every $t\in[0,T]$ and for
every $\phi\in\mathcal{V}_N$, we have
\begin{multline}\label{eqmultnesnnn}
\int_\O v_{0}^{(n)}(x)\cdot\phi(x)dx-\int_0^t\int_\O\nabla_x
u_n:\nabla_x\phi
+\int_0^t\int_\O F_n(u_n):\nabla_x\phi\\
=\int_\O u_n(x,t)\cdot\phi(x)dx\,.
\end{multline}
Since $|F_n(u_n)|\leq C|u_n|^2$, by \er{bonmdf},
\begin{equation}\label{jnh}
\|F_n(u_n(\cdot,t))\|_{L^1(\O,\R^{N\times N})}\leq \bar
C\quad\forall n\in\mathbb{N},t\in[0,T]\,.
\end{equation}
In particular $\{F_n(u_n)\}$ is bounded in
$L^1(\O\times(0,T),\R^{N\times N})$. Therefore, there exists a
finite Radon measure $\mu\in\mathcal{M}(\O\times(0,T),\R^{N\times
N})$, such that , up to a subsequence, $F_n(u_n)\rightharpoonup \mu$
weakly as a sequence of finite Radon measures. Then for every
$\psi\in C^\infty_0(\O\times(0,T),\R^{N\times N})$ we have
\begin{equation}\label{mu1}
\lim\limits_{n\to\infty}\int_0^T\int_\O F_n(u_n):\psi\,dxdt=
\int_{\O\times(0,T)}\psi :d\mu\,.
\end{equation}
Moreover, by \er{jnh}, we obtain
\begin{equation}\label{mu2}
|\mu|(\O\times(a,b))\leq\lim_{n\to\infty}\int_a^b\int_\O
|F_n(u_n)|\,dxdt \leq\bar C(b-a)\,.
\end{equation}
Then, by \er{jnh}, \er{mu1} and \er{mu2}, for every
$\phi\in\mathcal{V}_N$ and every $t\in[0,T]$ we obtain
\begin{equation}\label{mu3}
\lim\limits_{n\to\infty}\int_0^t\int_\O
F_n\big(u_n(x,s)\big):\nabla_x\phi(x)\,dxds=
\int_{\O\times(0,t)}\nabla_x\phi(x) :d\mu(x,s)\,.
\end{equation}
But since $u_n$ is bounded in $L^2(0,T;V_N)$, up to a subsequence,
it converge weakly in $L^2(0,T;V_N)$ to the limit $u_0$. We also
know that $u_n(\cdot,0)\rightharpoonup v_0(\cdot)$ weakly in $L_N$.
Plugging these facts and \er{mu3} into \er{eqmultnesnnn}, for every
$t\in[0,T]$ and every $\phi\in\mathcal{V}_N$ we infer
\begin{multline}\label{eqmultnesnnnlim}
\lim_{n\to\infty}\int_\O u_n(x,t)\cdot\phi(x)dx=\\ \int_\O
v_0(x)\cdot\phi(x)dx-\int_0^t\int_\O\nabla_x u_0:\nabla_x\phi
+\int_{\O\times(0,t)}\nabla_x\phi(x) :d\mu(x,s)\,.
\end{multline}
Since $\mathcal{V}_N$ is dense in $L_N$, by \er{bonmdf}, and
\er{eqmultnesnnnlim}, for every $t\in[0,T]$ there exists
$u(\cdot,t)\in L_N$ such that
\begin{equation}\label{waeae1kkd}
u_n(\cdot,t)\rightharpoonup u(\cdot,t) \quad\text{weakly in }
L_N\quad\forall t\in[0,T]\,,
\end{equation}
Moreover,
$\|u(\cdot,t)\|_{L_N}\leq C$. But we have $u_n\rightharpoonup u_0$
in $L^2(0,T;V_N)$, therefore $u=u_0$ and so $u\in L^2(0,T;V_N)\cap
L^\infty(0,T;L_N)$.
%
%
%
%
%
%
%
Then we can use \er{bonmdf}, \er{waeae1kkd} and Lemma \ref{ComTem1},
to deduce that $u_n\to u$ strongly in $L^2(0,T;L_N)$.
Then, up to a subsequence, we have $u_n(x,t)\to u(x,t)$ almost
everywhere in $\O\times(0,T)$. In particular
$f_n\big(|u_n(x,t)|^2\big)\to 1$ almost everywhere in
$\O\times(0,T)$. Then,
\begin{multline*}
\limsup\limits_{n\to\infty}\int_0^T\int_\O\big|f_n(|u_n|^2)(u_n\otimes
u_n)-(u\otimes u)\big|\,dxdt\leq
\\\limsup\limits_{n\to\infty}\int_0^T\int_\O|f_n(|u_n|^2)|\cdot|(u_n\otimes
u_n)-(u\otimes u)|,dxdt+ \limsup\limits_{n\to\infty}\int_0^T\int_\O
u^2\big|f_n(|u_n|^2)-1\big|\,dxdt=0\,.
\end{multline*}
Therefore, letting $n$ tend to $\infty$ in \er{nnhogfc48nnj}, we
obtain \er{nnhogfc48}.
Moreover, by \er{yteqinnnj1}, for a.e. $t\in[0,T]$ we obtain
\er{yteqin}.
This completes the proof.
%
%
%
%
%
%
%
%
\end{proof}

\section{Proof of the existence of solutions to \er{apprxnn}}\label{dva1}
Throughout this section we assume that $\O\subset\R^N$ is a bounded
domain.
The following Lemma can be proved in the same way as Theorem 4.1 in
\cite{Gal}.
\begin{lemma}\label{lem2}
Let $u\in L^2(0,T;V_N)\cap L^\infty(0,T;L_N)$ be such that
$\partial_t u\in L^2(0,T;V_N^{-1})$
and let $u(\cdot,t)$ be $L_N$-weakly continuous in $t$ on $[0,T]$
(see Lemma \ref{vbnhjjm}).
Consider $V_0(\cdot,t)\in L^2(0,T;V_N)$ as in Remark \ref{rem1},
corresponding to $\partial_t u$. Then for every $\tau\in [0,T]$ we
have
$$\int_0^\tau\int_\O\nabla_x u:\nabla_x V_0\,dxdt=\frac{1}{2}\bigg(\int_\O u^2(x,0)dx
-\int_\O u^2(x,\tau)dx\bigg)\,.$$
\end{lemma}
\begin{corollary}\label{corcor1}
Let $u\in L^2(0,T;V_N)$ be such that $\partial_t u\in
L^2(0,T;V_N^{-1})$. Then $u\in L^\infty(0,T;L_N)$.
\end{corollary}
We will give the proof of this Corollary in the Appendix.

 Next we have the second Corollary to Lemma \ref{lem2}.
\begin{corollary}\label{lem2'}
Let $F\in\mathfrak{F}$
and let $u\in L^2(0,T;V_N)\cap L^\infty(0,T;L_N)$ be such that
$\partial_t u\in L^2(0,T;V_N^{-1})$
and let $u(\cdot,t)$ be $L_N$-weakly continuous in $t$ on $[0,T]$
(see Lemma \ref{vbnhjjm}).
Consider $V(\cdot,t)\in L^2(0,T;V_N)$ as in Remark \ref{rem1},
corresponding to $\partial_t u+div_x\, F(u)$ (see Remark
\ref{rem3}). Then
%
%
%
%
%
%
for every $\tau\in [0,T]$ we have
\begin{equation}\label{cchin}
\int_0^\tau\int_\O\nabla_x u:\nabla_x
V\,dxdt=\frac{1}{2}\bigg(\int_\O u^2(x,0)dx -\int_\O
u^2(x,\tau)dx\bigg)\,.
\end{equation}
\end{corollary}
\begin{proof}
By Lemma \ref{lem2}, for every $\tau\in[0,T]$ we obtain
\begin{multline}\label{kklmn45}
\int_0^\tau\int_\O\nabla_x V:\nabla_x u\, dxdt-\int_0^\tau\int_\O
F(u):\nabla_x u\,dxdt\\ =\frac{1}{2}\bigg(\int_\O u^2(x,0)dx-\int_\O
u^2(x,\tau)dx\bigg)\,.
\end{multline}
%
%
%
%
%
%
%
%
But for almost every $t\in [0,T]$ $u(\cdot,t)\in V_N$, therefore,
for every such fixed $t$ there exists a sequence
$\{\delta_n(\cdot)\}_{n=1}^{\infty}\in\mathcal{V}_N$, such that
$\delta_n(\cdot)\to u(\cdot,t)$ in $V_N$. But for every
$\delta\in\mathcal{V}_N$ we obtain
$$\int_\O F(\delta):\nabla_x \delta=\int_\O\sum\limits_{i=1}^{N}\sum\limits_{j=1}^{N}:
F_{ij}(\delta)\frac{\partial\delta_i}{\partial
x_j}=\int_\O\sum\limits_{i=1}^{N}\sum\limits_{j=1}^{N}\frac{\partial
G_j}{\partial v_i}(\delta)\frac{\partial\delta_i}{\partial x_j}=
\int_\O div_x\, G(\delta)=0\,,$$ where $G$ is as in Remark
\ref{rem9888}. Therefore, since $F$ is Lipshitz function, we obtain
\begin{equation*}
\int_\O F(u(x,t)):\nabla_x
u(x,t)\,dx=\lim\limits_{n\to\infty}\int_\O F(\delta_n(x)):\nabla_x
\delta_n(x)\,dx=0\,.
\end{equation*}
%
%
%
%
%
%
%
%
Therefore, using \er{kklmn45}, we obtain \er{cchin} and the result
follows.
%
%
%
%
%
%
\end{proof}

\begin{definition}\label{defH}
Let $u\in L^2(0,T;V_N)\cap L^\infty(0,T;L_N)$ be such that
$\partial_t u\in L^2(0,T;V_N^{-1})$ and such that $u(\cdot,t)$ is
$L_N$-weakly continuous in $t$ on $[0,T]$. Denote the set of all
such functions $u$ by $\mathcal{R}$. For a fixed $F\in\mathfrak{F}$
and for every $u\in\mathcal{R}$ let $H_u(\cdot,t)\in L^2(0,T;V_N)$
be as in Remark \ref{rem1}, corresponding to $\partial_t u+div_x\,
F(u)$. That is for every $\psi(x,t)\in
C^\infty_c(\O\times(0,T),\R^N)$ such that $div_x\,\psi=0$ we have
\begin{equation*}
\int_0^T\int_\O\big(u\cdot\partial_t\psi+F(u):\nabla_x
\psi\big)dxdt=\int_0^T\int_\O\nabla_x H_u:\nabla_x\psi\,dxdt\,.
\end{equation*}
%
%
%
Define a functional $I_{F}(u):\mathcal{R}\to\R$ by
\begin{equation}\label{hgffck}
I_{F}(u):=\frac{1}{2}\bigg(\int_0^T\int_\O\Big(|\nabla_x u
|^2+|\nabla_x H_u
|^2\Big)dxdt+\int_\O|u(x,T)|^2dx\bigg)\,,
\end{equation}
and for every $v_0\in V_N$ consider the minimization problem
\begin{equation}\label{hgffckaaq1}
\inf\{I_{F}(u):\,u\in\mathcal{R},u(\cdot,0)=v_0(\cdot)\}\,.
\end{equation}
\end{definition}
\begin{remark}\label{rmmmdd}
Since by Corollary \ref{lem2'} we have
\begin{equation*}
\int_0^T\int_\O\nabla_x u:\nabla_x
H_u\,dxdt=\frac{1}{2}\bigg(\int_\O|u(x,0)|^2dx -\int_\O
|u(x,T)|^2dx\bigg)\,,
\end{equation*}
we can rewrite the definition of $I_F$ in \er{hgffck} by
\begin{equation}\label{fyhfypppp}
I_{F}(u):=\frac{1}{2}\bigg(\int_0^T\int_\O|\nabla_x u-\nabla_x
H_u|^2dxdt+\int_\O|u(x,0)|^2dx\bigg)\quad\quad\forall
u\in\mathcal{R}\,.
\end{equation}
\end{remark}
\begin{lemma}\label{EulerLagrange1}
For every $u\in\mathcal{R}$ and every $\delta(x,t)\in \mathcal{Y}$,
such that $\delta(x,0)=0$, we have
\begin{multline}\label{nolkj91}
\lim\limits_{s\to
0}\frac{I_{F}(u+s\delta)-I_{F}(u)}{s}=\\\int_0^T\int_\O\bigg\{\nabla_x
W_{u}:\nabla_x\delta+\partial_t\delta\cdot W_{u}
-\Big(\sum_{j=1}^{N}\delta_j\frac{\partial F}{\partial
u_j}(u)\Big):\nabla_x W_{u}\bigg\}\,dxdt\,,
\end{multline}
where we denote $W_{u}:=u-H_u$.
\end{lemma}
\begin{proof}
%
%
%
%
%
%
We have
\begin{multline}\label{nolkj3}
\frac{1}{2s}\int_0^T\int_\O\Big(|\nabla_x
W_{(u+s\delta)}|^2-|\nabla_x W_{u}|^2\Big)=\\
\frac{1}{2s}\int_0^T\int_\O\big(\nabla_x W_{(u+s\delta)}-\nabla_x
W_{u}\big):\big(\nabla_x W_{(u+s\delta)}+\nabla_x
W_{u}\big)=\\\frac{1}{2s}\int_0^T
\Big<\Big(s\cdot\partial_t\delta-s\Delta_x\delta+div_x\,
F(u+s\delta)-div_x\,
F(u)\Big)(\cdot,t),\big(W_{(u+s\delta)}+W_{u}\big)(\cdot,t)\Big>\,dt
\\=\int_0^T\int_\O\bigg\{\frac{1}{2}\big(\nabla_x W_{(u+s\delta)}+\nabla_x
W_{u}\big):\nabla_x\delta+\partial_t\delta(x,t)
\cdot\frac{1}{2}\big(W_{(u+s\delta)}(x,t)+W_{u}(x,t)\big)\bigg\}\\
-\int_0^T\int_\O\frac{1}{s}\big(F(u+s\delta)-
F(u)\big):\frac{1}{2}\big(\nabla_x W_{(u+s\delta)}+\nabla_x
W_{u}\big)\,dxdt\,.
\end{multline}
Since $F$ is Lipschitz and $C^1$, we obtain
\begin{equation}\label{nolkj4}
\frac{1}{s}\big(F(u+s\delta)-
F(u)\big)\to\sum_{j=1}^{N}\delta_j\frac{\partial F}{\partial
u_j}(u)\quad\text{as }s\to 0\quad\text{strongly in
}L^2(\O\times(0,T),\R^{N\times N})\,.
\end{equation}
On the other hand, for every $h(x,t)\in L^2(0,T;V_N)$ we obtain
\begin{multline}\label{nolkj5}
\lim\limits_{s\to 0}\int_0^T\int_\O\big(\nabla_x
W_{(u+s\delta)}-\nabla_x W_{u}\big):\nabla_x h(x,t)=\\
\lim\limits_{s\to
0}\bigg(s\int_0^T\int_\O(\partial_t\delta-\Delta_x\delta)\cdot
h\,dxdt-\int_0^T\int_\O\big(F(u+s\delta)-F(u)\big):\nabla_x
h\,dxdt\bigg)=0\,.
\end{multline}
Therefore
\begin{equation}\label{nolkj6}
W_{(u+s\delta)}\rightharpoonup W_u\quad\text{weakly in
}L^2(0,T;V_N)\,.
\end{equation}
In particular $W_{(u+s\delta)}$ remains bounded in $L^2(0,T;V_N)$ as
$s\to 0$. Therefore, by \er{nolkj3}, we obtain $$\lim\limits_{s\to
0}\int_0^T\int_\O\big(|\nabla_x W_{(u+s\delta)}|^2-|\nabla
W_{u}|^2\big)\,dxdt=0\,.$$ So
\begin{equation}\label{nolkj7}
W_{(u+s\delta)}\to W_u\quad\text{strongly in }L^2(0,T;V_N)\,.
\end{equation}
Therefore, using \er{nolkj7} and \er{nolkj4} in \er{nolkj3}, we
infer
\begin{multline}\label{nolkj8}
\lim\limits_{s\to 0}\frac{1}{2s}\int_0^T\int_\O\Big(|\nabla_x
W_{(u+s\delta)}|^2-|\nabla_x W_{u}|^2\Big)=
\\
\int_0^T\int_\O\bigg\{\nabla_x
W_{u}:\nabla_x\delta+\partial_t\delta\cdot W_{u}
-\Big(\sum_{j=1}^{N}\delta_j\frac{\partial F}{\partial
u_j}(u)\Big):\nabla_x W_{u}\bigg\}\,dxdt\,.
\end{multline}
So, by \er{fyhfypppp} and \er{nolkj8}, we obtain that for every
$\delta(x,t)\in \mathcal{Y}$, such that $\delta(x,0)=0$, we must
have \er{nolkj91}.
\end{proof}
\begin{lemma}\label{EulerLagrange}
Let $u\in \mathcal{R}$ be a minimizer to \er{hgffckaaq1}. Then
$H_u=u$, i.e.
$$\Delta_x u=\partial_t u+div_x\,F(u)+\nabla_x p\,.$$
\end{lemma}
\begin{proof}
Let $\delta(x,t)\in \mathcal{Y}$ be such that $\delta(x,0)=0$. Then
for every $s\in\R$ $(u+s\delta)\in\mathcal{R}$ and
$(u+s\delta)(\cdot,0)=v_0(\cdot)$. Therefore,
\begin{equation}\label{nolkj}
\lim\limits_{s\to 0}\frac{I_{F}(u+s\delta)-I_{F}(u)}{s}=0\,.
\end{equation}
So, by \er{nolkj91} in Lemma \ref{EulerLagrange1}, for every
$\delta(x,t)\in \mathcal{Y}$ such that $\delta(x,0)=0$ we must have
\begin{equation}\label{nolkj989}
\int_0^T\int_\O\bigg\{\nabla_x
W_{u}:\nabla_x\delta+\partial_t\delta\cdot W_{u}
-\Big(\sum_{j=1}^{N}\delta_j\frac{\partial F}{\partial
u_j}(u)\Big):\nabla_x W_{u}\bigg\}\,dxdt=0\,,
\end{equation}
where $W_u\in L^2(0,T;V_N)$ defined by $W_u=u-H_u$.
Since $\frac{\partial F}{\partial u_j}\in L^\infty$, we obtain that
the functional $L(\phi):V_N\to\R$ defined by
$$L(\phi):=\int_\O\Big(\sum_{j=1}^{N}\phi_j\frac{\partial
F}{\partial u_j}(u)\Big):\nabla_x W_{u}\,dx$$ belongs to $V_N^{-1}$
for a.e. $t\in(0,T)$. Moreover there exists $Q(x,t)\in L^2(0,T;V_N)$
such that for a.e. $t\in(0,T)$ we have
$$L(\phi):=\int_\O\Big(\sum_{j=1}^{N}\phi_j\frac{\partial
F}{\partial u_j}(u)\Big):\nabla_x W_{u}\,dx=\int_\O\nabla_x
Q(x,t):\nabla_x\phi(x)\,dx\quad\forall \phi\in V_N\,.$$ Then from
\er{nolkj989} we obtain that $\partial_t W_u\in L^2(0,T;V_N^{-1})$
and we have
\begin{multline}\label{eqnmmxx}
<\partial_t
W_u(\cdot,\cdot),\psi(\cdot,\cdot)>=-\int_0^T\int_\O\nabla_x(Q-W_u):\nabla_x\psi\,dxdt
\\ \forall\psi\in C_c^\infty(\O\times(0,T),\R^N)\text{ s.t.
}div_x\,\psi=0\,.
\end{multline}
Therefore, by Corollary \ref{corcor1} $W_u\in L^\infty(0,T;L_N)$ and
by Lemma \ref{vbnhjjm} we can redefine $W_u(\cdot,t)$ on a set of
Lebesgue measure zero on $[0,T]$ so that $W_u(\cdot,t)$ be
$L_N$-weakly continuous in $t$ on $[0,T]$. From now we consider such
$W_u$. Moreover, by \er{eqmultnep} and \er{eqnmmxx}, for every
$\delta\in\mathcal{Y}$, such that $\delta(x,0)=0$, we obtain
\begin{multline*}
\int_0^T\int_\O\nabla_x (Q-W_u):\nabla_x\delta\,
dxdt-\int_0^T\int_\O W_u\cdot\partial_t\delta\,dxdt =-\int_\O
W_u(x,T)\cdot\delta(x,T)dx\,,
\end{multline*}
or in the another form
\begin{multline}\label{eqmultnepche}
\int_0^T\int_\O\nabla_x W_u:\nabla_x\delta\,
dxdt-\int_0^T\int_\O\Big(\sum_{j=1}^{N}\delta_j\frac{\partial
F}{\partial u_j}(u)\Big):\nabla_x W_{u}\,dxdt\\+\int_0^T\int_\O
W_u\cdot\partial_t\delta\,dxdt -\int_\O
W_u(x,T)\cdot\delta(x,T)dx=0\,.
\end{multline}
Comparing \er{eqmultnepche} with \er{nolkj989}, we obtain that
$W_u(\cdot,T)=0$. Therefore, by Corollary \ref{corcor1} and Lemma
\ref{lem2},
for every $t\in[0,T]$ we obtain
$$\int_t^T\int_\O\nabla_x W_u:\nabla_x (Q-W_u)\,dxds=\frac{1}{2}\int_\O W_u^2(x,t)dx
\,,$$
or in the equivalent form
\begin{equation}\label{eqbbbkl}
\int_t^T\int_\O|\nabla_x W_u|^2\,dxds+\frac{1}{2}\int_\O
W_u^2(x,t)dx
=\int_t^T\int_\O\Big(\sum_{j=1}^{N}\,(W_u)_j\,\frac{\partial
F}{\partial u_j}(u)\Big):\nabla_x W_{u}\,dxds\,.
\end{equation}
In particular there exists $C>0$, independent of $t$, such that
\begin{multline*}
\int_t^T\int_\O|\nabla_x W_u|^2\,dxds\leq
\int_t^T\int_\O\Big(\sum_{j=1}^{N}\,(W_u)_j\,\frac{\partial
F}{\partial u_j}(u)\Big):\nabla_x W_{u}\,dxds\\ \leq
C\bigg(\int_t^T\int_\O|\nabla_x
W_u|^2\,dxds\,\cdot\,\int_t^T\int_\O|W_u|^2\,dxds\bigg)^{1/2}\,.
\end{multline*}
So
\begin{equation}\label{fgfd}
\int_t^T\int_\O|\nabla_x W_u|^2\,dxds\leq C^2\int_t^T\int_\O
|W_u|^2\,dxds\,.
\end{equation}
Then, using \er{eqbbbkl} and \er{fgfd} we obtain
\begin{multline}\label{fcvvvvv}
\frac{1}{2}\int_\O W_u^2(x,t)dx\leq
\int_t^T\int_\O\Big(\sum_{j=1}^{N}\,(W_u)_j\,\frac{\partial
F}{\partial u_j}(u)\Big):\nabla_x W_{u}\,dxds\\ \leq
C\bigg(\int_t^T\int_\O|\nabla_x
W_u|^2\,dxds\,\cdot\,\int_t^T\int_\O|W_u|^2\,dxds\bigg)^{1/2}\leq
C^2\int_t^T\int_\O|W_u|^2\,dxds\,.
\end{multline}
%
%
%
%
%
%
%
%
Then by Gronwall's Lemma $\int_\O W_u^2(x,t)dx=0$. So, by definition
of $W_u$ we obtain $H_u=u$. This completes the proof.
\end{proof}
\begin{theorem}\label{premain}
For every $v_0(\cdot)\in V_N$ there exists a minimizer $u$ to
\er{hgffckaaq1}. It satisfies $H_u=u$, i.e.
$$\Delta_x u=\partial_t u+div_x\,F(u)+\nabla_x p\,,$$
$u(x,0)=v_0(x)$ and
\begin{equation}\label{yteq}
\frac{1}{2}\int_\O u^2(x,\tau)dx+\int_0^\tau\int_\O|\nabla_x
u|^2\,dxdt
=\frac{1}{2}\int_\O v_0^2(x)dx\quad\quad\quad\forall \tau\in[0,T]\,.
\end{equation}
Moreover if $v\in \mathcal{R}$ satisfy $v(\cdot,0)=v_0(\cdot)$ and
$H_v=v$, i.e. $\Delta_x v=\partial_t v+div_x\,F(v)+\nabla_x p$, then
$v$ is a minimizer to \er{hgffckaaq1}.
\end{theorem}
\begin{proof}
First of all we want to note that the set
$A_{v_0}:=\{u\in\mathcal{R}:\,u(\cdot,0)=v_0(\cdot)\}$ is not empty.
In particular the function $u_0(\cdot,t):=v_0(\cdot)$ belongs to
$A_{v_0}$. Let
$$K:=\inf\limits_{u\in A_{v_0}}I_{F}(u)\,.$$
Then $K\geq 0$. Consider the minimizing sequence $\{u_n\}\subset
A_{v_0}$, i.e. the sequence such that
$\lim_{n\to\infty}I_{F}(u_n)=K$. Then, by the definition of $I_{F}$
in \er{hgffck}, we obtain that there exists $C>0$, independent of
$n$, such that
\begin{equation}\label{oteqm9999}
\int_0^T\int_\O\big(|\nabla_x u_n|^2+|\nabla_x
H_{u_n}|^2\big)\,dxdt\leq C\,.
\end{equation}
Then
up to a
subsequence,
\begin{multline}\label{staeae1glll}
u_n\rightharpoonup u_0 \quad\text{weakly in } L^2(0,T;V_N)\quad
\text{and } H_{u_n}\rightharpoonup\bar H\quad\text{weakly in }
L^2(0,T;V_N)\,.
\end{multline}
From the other hand, by Corollary \ref{lem2'}, for every $t\in
[0,T]$ we have
\begin{equation*}
\int_\O u_n^2(x,t)dx=\int_\O u_n^2(x,0)dx- 2\int_0^t\int_\O\nabla_x
u_n:\nabla_x H_{u_n} \,.
\end{equation*}
Therefore, since, $u_n$ and $H_{u_n}$ are bounded in $L^2(0,T;V_N)$
by \er{oteqm9999} and $u_n(\cdot,0)=v_0(\cdot)$
we obtain that there exists $C>0$ independent of $n$ and $t$ such
that
\begin{equation}\label{bonmd}
\|u_n(\cdot,t)\|_{L_N}\leq C\quad\forall n\in\mathbb{N}, t\in[0,T].
\end{equation}
In particular, up to a further subsequence
$F(u_n)\rightharpoonup\bar F$ weakly in
$L^2(\O\times(0,T),\R^{N\times N})$. Then by \er{staeae1glll} and
\er{eqmultnep}, for every $t\in[0,T]$ and for every
$\phi\in\mathcal{V}_N$, we have
\begin{multline}\label{eqmultnesnnndds}
\lim\limits_{n\to\infty}\int_\O
u_n(x,t)\cdot\phi(x)dx=\\\lim\limits_{n\to\infty}\Bigg(\int_\O
u_n(x,0)\cdot\phi(x)dx-\int_0^t\int_\O\nabla_x H_{u_n}:\nabla_x\phi
+\int_0^t\int_\O F(u_n):\nabla_x\phi\Bigg)\\=\int_\O
v_0(x)\cdot\phi(x)dx-\int_0^t\int_\O\nabla_x \bar H:\nabla_x\phi
+\int_0^t\int_\O\bar F:\nabla_x\phi\,.
\end{multline}
Since $\mathcal{V}_N$ is dense in $L_N$, by \er{bonmd}, and
\er{eqmultnesnnndds}, for every $t\in[0,T]$ there exists
$u(\cdot,t)\in L_N$ such that
\begin{equation}\label{waeae1}
u_n(\cdot,t)\rightharpoonup u(\cdot,t) \quad\text{weakly in }
L_N\quad\forall t\in[0,T]\,.
\end{equation}
Moreover,
$\|u(\cdot,t)\|_{L_N}\leq C$. But we have $u_n\rightharpoonup u_0$
weakly in $L^2(0,T;L_N)$, therefore $u=u_0$ a.e. and so $u\in
L^2(0,T;V_N)\cap L^\infty(0,T;L_N)$. Then using \er{bonmd},
\er{staeae1glll}, \er{waeae1} and Lemma \ref{ComTem1} we deduce that
\begin{equation}\label{strtrtrhhhk}
u_n\to u\quad \text{strongly in }L^2(0,T;L_N)\,.
\end{equation}
Moreover, by \er{eqmultnesnnndds} we obtain that $u(\cdot,t)$ is
$L_N$-weakly continuous in $t$ on $[0,T]$.
Therefore, by \er{waeae1} and \er{staeae1glll},
\begin{multline}\label{muuytkkk}
\int_0^T\int_\O|\nabla_x
u|^2\,dxdt+\int_\O|u(x,T)|^2dx\leq\lim\limits_{n\to\infty}\bigg(\int_0^T\int_\O|\nabla_x
u_n|^2\,dxdt+\int_\O|u_n(x,T)|^2dx\bigg)\,.
\end{multline}
Next for every $\psi(x,t)\in C^\infty_c(\O\times(0,T),\R^N)$ such
that $div_x\,\psi=0$ we obtain
\begin{multline}\label{nnhogfc}
\lim\limits_{n\to\infty}\int_0^T\int_\O
\big(u_n\cdot\partial_t\psi+F(u_n):\nabla_x
\psi\big)dxdt=\\\lim\limits_{n\to\infty}\int_0^T\int_\O\nabla_x
H_{u_n}:\nabla_x\psi\, dxdt=\int_0^T\int_\O\nabla_x \bar
H:\nabla_x\psi\,dxdt\,.
\end{multline}
But since $F$ is a Lipschitz function, by \er{strtrtrhhhk} we obtain
$$\lim\limits_{n\to\infty}\int_0^T\int_\O
\big(u_n\cdot\partial_t\psi+F(u_n):\nabla_x
\psi\big)dxdt=\int_0^T\int_\O
\big(u\cdot\partial_t\psi+F(u):\nabla_x \psi\big)dxdt$$ So, by
\er{nnhogfc}, for every $\psi(x,t)\in
C^\infty_c(\O\times(0,T),\R^N)$ such that $div_x\,\psi=0$ we deduce
\begin{equation}\label{nnhogfc1}
\int_0^T\int_\O\big(u\cdot\partial_t\psi+F(u):\nabla_x
\psi\big)dxdt=\int_0^T\int_\O\nabla_x \bar H:\nabla_x\psi\,dxdt\,.
\end{equation}
In particular $\partial_t u+div_x\, F(u)\in L^2(0,T;V_N^{-1})$.
Therefore $\partial_t u\in L^2(0,T;V_N^{-1})$ and then $u\in
A_{v_0}=\{u\in\mathcal{R}:\,u(\cdot,0)=v_0(\cdot)\}$. Moreover, by
\er{nnhogfc1}, we obtain that $H_u=\bar H$. So, as before,
\begin{equation}\label{muuytkkk888}
\int_0^T\int_\O|\nabla_x
H_u|^2\,dxdt\leq\lim\limits_{n\to\infty}\int_0^T\int_\O|\nabla_x
H_{u_n}|^2\,dxdt\,.
\end{equation}
Combining \er{muuytkkk888} with \er{muuytkkk}, we infer
$$I_{F}(u)\leq\lim_{n\to\infty}I_{F}(u_n)=K\,.$$
Therefore, $u$ is a minimizer to \er{hgffckaaq1}. By Lemma
\ref{EulerLagrange} it satisfies $H_u=u$, i.e.
$$\Delta_x u=\partial_t u+div_x\,F(u)+\nabla_x p\,.$$
Moreover, by Lemma \ref{lem2}, for every $t\in [0,T]$ we have
$$\int_0^t\int_\O\nabla_x u:\nabla_x H_{u}\,
=\frac{1}{2}\bigg(\int_\O v_0^2(x)dx-\int_\O u^2(x,t)dx\bigg)\,.$$
Therefore we obtain \er{yteq}. Moreover,
$I_{F}(u)=\frac{1}{2}\int_\O v_0^2(x)dx$.
Finally if $v\in \mathcal{R}$ satisfy $v(\cdot,0)=v_0(\cdot)$ and
$H_v=v$ then by \er{fyhfypppp} we have $I_{F}(v)=\frac{1}{2}\int_\O
v_0^2(x)dx=I_{F}(u)$.
So $v$ is a minimizer to \er{hgffckaaq1}.
\end{proof}
\begin{remark}
For a fixed $r(x,t)\in L^2(0,T;V_N)$ we can define a functional
$\bar I_{\{F,r\}}(u):\mathcal{R}\to\R$ by
\begin{equation}\label{hgffckwithr}
\bar I_{\{F,r\}}(u):=\frac{1}{2}\bigg(\int_0^T\int_\O\Big(|\nabla_x
u+\nabla_x r|^2+|\nabla_x H_u-\nabla_x
r|^2\Big)dxdt+\int_\O|u(x,T)|^2dx\bigg)\,,
\end{equation}
and for every $v_0\in V_N$ we can consider the minimization problem
\begin{equation}\label{hgffckaaq1withr}
\inf\{\bar
I_{\{F,r\}}(u):\,u\in\mathcal{R},u(\cdot,0)=v_0(\cdot)\}\,.
\end{equation}
Then similarly to the proof of Theorem \ref{premain} we can prove
that there exists a minimizer $u$ to \er{hgffckaaq1withr} and it
satisfies $H_u=u+r$, i.e.
$$\Delta_x u+\Delta_x r=\partial_t u+div_x\,F(u)+\nabla_x p\,.$$
Then, using this fact, as in the proof of Theorem
\ref{WeakNavierStokes} we can deduce the existence of a weak
solution to \er{IBNS} with $f\in L^2(0,T;V^{-1}_N)$.
\end{remark}
\begin{remark}
Similar method as in the proof of Theorem \ref{WeakNavierStokes} we
can apply to the unbounded domain $\O$. In this case we consider a
sequence of smooth bounded domains $\{\O_n\}$, such that
$\O_n\subset\O_{n+1}$ and $\bigcup_{n=1}^{\infty}\O_n=\O$, and a
sequence $v_0^{(n)}\to v_0$ in $L_N$, such that $\supp
v_0^{(n)}\subset\O_n$. Consider $u_n(x,t)\in\mathcal{R}(\O_n)$, such
that $u_n(\cdot,0)=v_{0}^{(n)}(\cdot)$ and  for every $\psi(x,t)\in
C^\infty_c(\O_n\times(0,T),\R^N)$, satisfying $div_x\,\psi=0$, we
have \er{nnhogfc48nnj}, where $F_n$ is defined by \er{klnlnkmlm}.
Then
we can deduce that there exists $u\in L^2(0,T;V_N)\cap
L^\infty(0,T;L_N)$ such that, up to a subsequence, $u_n\to u$
strongly in $L^2_{loc}(\O\times(0,T),\R^N)$. Then $u$ will satisfy
conditions of Theorem \ref{WeakNavierStokes}.
\end{remark}

\section{Variational principle for more regular solutions of the Navier-Stokes
Equations}\label{odin} Let $\O\subset\R^N$ be a domain with
Lipschitz boundary (not necessarily bounded). We denote by $H_N$ the
closure of $\mathcal{V}_N$ in $H^1_0(\O,\R^N)$ (the spaces $H_N$ and
$V_N$ differ only in the case of unbounded domain). For every $u\in
L^4(\O\times(0,T),\R^N)$ we have $(u\otimes u)\in
L^2\big(0,T;L^2(\O,\R^{N\times N})\big)$ and therefore
$div_x\,(u\otimes u)\in L^2(0,T;V_N^{-1})$. If in addition
$\partial_t u\in L^2(0,T;V_N^{-1})$ then we obtain $\partial_t
u+div_x\, (u\otimes u)\in L^2(0,T;V_N^{-1})$.
\begin{definition}\label{defH1}
Let $u\in L^2(0,T;H_N)\cap L^\infty(0,T;L_N)$ be such that
$\partial_t u\in L^2(0,T;V_N^{-1})$ and such that $u(\cdot,t)$ is
$L_N$-weakly continuous in $t$ on $[0,T]$. Denote the set of all
such functions $u$ by $\mathcal{R'}$. Denote the set
$\mathcal{R'}\cap L^4(\O\times(0,T),\R^N)$ by $\mathcal{P}$. For
every $u\in\mathcal{P}$ let $\bar H_u(\cdot,t)\in L^2(0,T;V_N)$ be
as in Remark \ref{rem1}, corresponding to $\partial_t u+div_x\,
(u\otimes u)$. That is for every $\psi(x,t)\in
C^\infty_c(\O\times(0,T),\R^N)$ such that $div_x\,\psi=0$ we have
\begin{equation*}
\int_0^T\int_\O\big(u\cdot\partial_t\psi+(u\otimes u):\nabla_x
\psi\big)dxdt=\int_0^T\int_\O\nabla_x \bar
H_u:\nabla_x\psi\,dxdt\,.
\end{equation*}
For a fixed $r(x,t)\in L^2(0,T;V_N)$ define a functional
$J_{\{\f,r\}}(u):\mathcal{P}\to\R$ by
\begin{equation}\label{hgffck1}
J_{\{\f,r\}}(u):=\frac{1}{2}\bigg(\int_0^T\int_\O\Big(|\nabla_x
u+\nabla_x r|^2+|\nabla_x \bar H_u-\nabla_x
r|^2\Big)dxdt+\int_\O|u(x,T)|^2dx\bigg)\,.
\end{equation}
\end{definition}
\begin{theorem}\label{vari}
Let $v_0\in L_N$ and $r(x,t)\in L^2(0,T;V_N)$. Assume that there
exists $u\in\mathcal{P}$ which satisfies $u(x,0)=v_0(x)$, and
\begin{equation}\label{nnhogfc481}
\int_0^T\int_\O\big(u\cdot\partial_t\psi+(u\otimes u):\nabla_x
\psi\big)dxdt=\int_0^T\int_\O(\nabla_x u+\nabla_x
r):\nabla_x\psi\,dxdt\,
\end{equation}
for every $\psi(x,t)\in C^\infty_c(\O\times(0,T),\R^N)$, such that
$div_x\,\psi=0$, i.e.
$$\Delta_x u=\partial_t u+div_x\,(u\otimes u)+\nabla_x p-\Delta_x r\,.$$
Then $u$ is a minimizer of the following problem
\begin{equation}\label{hgffckaaq1111}
\inf\{J_{\{\f,r\}}(u):\,u\in\mathcal{P},u(\cdot,0)=v_0(\cdot)\}\,.
\end{equation}
Moreover if $\bar u$ is a minimizer to \er{hgffckaaq1111}, then
$\bar u$ is a solution to \er{nnhogfc481}.
\end{theorem}
\begin{proof}
In the same way as in the proof of Theorem 4.1 in \cite{Gal} we
obtain that for every $\bar u\in \mathcal{P}$ we must have
\begin{equation}\label{cchin111}
\int_0^T\int_\O\nabla_x \bar u:\nabla_x H_{\bar
u}\,dxdt=\frac{1}{2}\bigg(\int_\O\bar u^2(x,0)dx -\int_\O\bar
u^2(x,T)dx\bigg)\,.
\end{equation}
Therefore,
\begin{equation}\label{cchin111111}
J_{\{\f,r\}}(\bar u)=\frac{1}{2}\int_0^T\int_\O\big(|\nabla_x \bar
u+\nabla_x r-\nabla_x H_{\bar u}|^2+|\nabla_x
r|^2\big)\,dxdt+\frac{1}{2}\int_\O\bar u^2(x,0)dx\,.
\end{equation}
Therefore, $u\in\mathcal{P}$ which satisfy $u(x,0)=v_0(x)$ and
$\nabla_x u+\nabla_x r=\nabla_x H_{u}$ will be the minimizer to
\er{hgffckaaq1111}. Then also every minimizer $\bar u$ will satisfy
$\nabla_x \bar u+\nabla_x r=\nabla_x H_{\bar u}$, i.e. will satisfy
\er{nnhogfc481}.
\end{proof}

\appendix
\section{}
\begin{proof}[Proof of Lemma \ref{ComTem1}]
%
%
%
%
%
%
%
%
Since for every $n=0,1,\ldots$ and every $t\in(0,T)$ the functional
$l_{n,t}(\phi):=\int_\O u_n(x,t)\cdot \phi(x)\,dx$ is bounded in
$H^{1}_0(\O,\R^N)$, by Rietz Representation Theorem for every
$n=0,1,\ldots$ and every $t\in[0,T]$ there exists $w_n(\cdot,t)\in
H^{1}_0(\O,\R^N)$ such that
\begin{equation}\label{vhsjklppp}
l_{n,t}(\phi)=\int_\O u_n(x,t)\cdot \phi(x)\,dx=\int_\O\nabla_x
w_n(x,t):\nabla_x\phi(x)\,dx\quad\forall\phi\in H^{1}_0(\O,\R^N)\,,
\end{equation}
Equation \er{vhsjklppp} gives in particular
\begin{equation}\label{vhsjklpppvbvbhbjh}
\int_\O u_n(x,t)\cdot w_n(x,t)\,dx=\int_\O|\nabla_x
w_n(x,t)|^2\,dx\,.
\end{equation}
Then we obtain that there exist $C_0,C>0$, independent of $n$ and
$t$, such that
\begin{equation}\label{vhsjklppr}
\int_\O|\nabla_x w_n(x,t)|^2\,dx\leq C_0\int_\O
|u_n(x,t)|^2\,dx\leq C\,.
\end{equation}
Moreover, using \er{waeae12}, \er{vhsjklpppvbvbhbjh} and the compact
embedding of $H^{1}_0(\O,\R^N)$ into $L^2(\O,\R^N)$, we obtain
\begin{equation}\label{vhsjklppdtr}
w_n(\cdot,t)\to w_0(\cdot,t)\quad\text{strongly in }
H^{1}_0(\O,\R^N)\quad\forall t\in(0,T)\,.
\end{equation}
We have $w_n(\cdot,\cdot)\in L^2(0,T;H^{1}_0(\O,\R^N))$ and
moreover, by \er{vhsjklppdtr} and \er{vhsjklppr}, we obtain
\begin{equation}\label{vhsjklppdtrgl}
w_n(\cdot,\cdot)\to w_0(\cdot,\cdot)\quad\text{strongly in }
L^2(0,T;H^{1}_0(\O,\R^N))\,.
\end{equation}
But, by \er{vhsjklppp}, we have,
\begin{equation}\label{vhsjklpddd}
\int_0^T\int_\O|u_n(x,t)|^2\,dxdt=\int_0^T\int_\O\nabla_x
w_n(x,t):\nabla_x u_n(x,t)\,dxdt\,,
\end{equation}
and since $u_n\rightharpoonup u_0$ in $L^2(0,T;V_N)$, by
\er{vhsjklppdtrgl}, we obtain
\begin{multline}\label{vhsjklpdddaa}
\lim\limits_{n\to\infty}\int_0^T\int_\O
|u_n(x,t)|^2\,dxdt=\lim\limits_{n\to\infty}\int_0^T\int_\O\nabla_x
w_n(x,t):\nabla_x u_n(x,t)\,dxdt\\=\int_0^T\int_\O\nabla_x
w_0(x,t):\nabla_x u_0(x,t)\,dxdt= \int_0^T\int_\O|u_0(x,t)|^2\,dxdt
\,.
\end{multline}
But by \er{waeae1stggg} we have
\begin{equation}\label{waeae1glll}
u_n\rightharpoonup u_0 \quad\text{weakly in } L^2(0,T;L_N)\,.
\end{equation}
Therefore, by \er{waeae1glll} and \er{vhsjklpdddaa} we obtain
\er{staeae1gllluk}.
\end{proof}
\begin{proof}[Proof of Corollary \ref{corcor1}]
Let $\eta\in C^\infty_c(\R,\R)$ be a mollifying kernel, satisfying
$\eta\geq 0$, $\int_\R\eta(t)dt=1$, $\supp\eta\subset[-1,1]$ and
$\eta(-t)=\eta(t)$ $\forall t$. Given small $\e>0$ and $\psi(x,t)\in
C^\infty_c(\O\times(2\e,T-2\e),\R^N)$ such that $div_x\,\psi=0$,
define
\begin{equation}\label{molkl}
\psi_\e(x,t):=\frac{1}{\e}\int\limits_0^{T}\eta\Big(\frac{s-t}{\e}\Big)
\psi(x,s)ds\,.
\end{equation}
Then $\psi_\e(x,t)\in C^\infty_c(\O\times(0,T),\R^N)$ and satisfies
$div_x\,\psi_\e=0$. Therefore we obtain
\begin{equation}\label{bnn898989899988}
\int_0^T\int_\O u\cdot\partial_t\psi_\e\,dxdt=\int_0^T\int_\O
\nabla_x V_u:\nabla_x\psi_\e\,dxdt\,,
\end{equation}
where $V_u(\cdot,t)\in L^2(0,T;V_N)$ is as in Remark \ref{rem1},
corresponding to $\partial_t u$. But
\begin{multline*}
\int_0^T\int_\O u\cdot\partial_t\psi_\e\,dxdt= \int_0^T\int_\O
u(x,t)\cdot\bigg(\frac{1}{\e}\int_0^{T}\eta\Big(\frac{s-t}{\e}\Big)
\partial_s\psi(x,s)ds\bigg)\,dxdt=\\
\int_0^T\int_\O
\partial_t \psi(x,t)\cdot\bigg(\frac{1}{\e}\int_0^{T}\eta\Big(\frac{s-t}{\e}\Big)
u(x,s)ds\bigg)\,dxdt=\int_0^T\int_\O
\partial_t \psi(x,t)\cdot u_\e(x,t)\,dxdt\,,
\end{multline*}
where $u_\e(x,t)=\frac{1}{\e}\int_0^{T}\eta((s-t)/\e) u(x,s)ds$. By
the other hand
\begin{multline*}
\int_0^T\int_\O\nabla_x V_u:\nabla_x\psi_\e\,dxdt= \int_0^T\int_\O
\nabla_x
V_u(x,t):\bigg(\frac{1}{\e}\int_0^{T}\eta\Big(\frac{s-t}{\e}\Big)
\nabla_x\psi(x,s)ds\bigg)\,dxdt\\
=\int_0^T\int_\O
\nabla_x\psi(x,t):\nabla_x\bigg(\frac{1}{\e}\int_0^{T}\eta\Big(\frac{s-t}{\e}\Big)
V_u(x,s)ds\bigg)\,dxdt\\=\int_0^T\int_\O\nabla_x\psi(x,t):
\nabla_x(V_u)_\e(x,t)\,dxdt\,,
\end{multline*}
where $(V_u)_\e(x,t)=\frac{1}{\e}\int_0^{T}\eta((s-t)/\e)
V_u(x,s)ds$. Therefore, by \er{bnn898989899988}, we infer
\begin{equation}\label{bnn8989898}
\int_0^T\int_\O u_\e\cdot\partial_t\psi\,dxdt=\int_0^T\int_\O
\nabla_x (V_u)_\e:\nabla_x\psi\,dxdt\,.
\end{equation}
So $\partial_t u_\e\in L^2(2\e,T-2\e;V_N^{-1})$. Moreover $u_\e \in
L^2(0,T;V_N)\cap L^\infty(0,T;L_N)$. We have $u_\e\to u$ and
$(V_u)_\e\to V_u$ strongly in $L^2(0,T;V_N)$ as $\e\to 0$. Moreover,
up to a subsequence $\e_n\to 0$, we have $u_{\e_n}(\cdot,t)\to
u(\cdot,t)$ strongly in $L_N$ a.e. in $[0,T]$. In addition, by Lemma
\ref{lem2}, for every $a,b\in [2\e, T-2\e]$ we have
\begin{equation}\label{nnbbvv888}
\int_a^b\int_\O\nabla_x u_\e:\nabla_x (V_u)_\e\,dxdt=\frac{1}{2}
\bigg(\int_\O u_\e^2(x,a)dx -\int_\O u_\e^2(x,b)dx\bigg)\,.
\end{equation}
Then letting $\e\to 0$ in \er{nnbbvv888}, we obtain that for almost
every $a$ and $b$ in $(0,T)$ we have
$$\int_a^b\int_\O\nabla_x u:\nabla_x V_u\,dxdt=\frac{1}{2}
\bigg(\int_\O u^2(x,a)dx -\int_\O u^2(x,b)dx\bigg)\,.
$$
So $u\in L^\infty(0,T;L_N)$.
\end{proof}

%
%
%
%
%
%
%
%
\end{document}